\documentclass[12pt]{amsart}
\usepackage{amssymb,latexsym} 
\usepackage{fullpage}
\usepackage{enumerate}

\makeatletter
\@namedef{subjclassname@2010}{%
\textup{2010} Mathematics Subject Classification}
\makeatother

\newtheorem{thm}{Theorem}

\newtheorem{lem}[thm]{Lemma}

\newtheorem{prop}[thm]{Proposition}

\theoremstyle{definition}

\numberwithin{equation}{section}

\def\Z{\mathbb Z}
\def\M{\mathcal M}

\def\F{\mathcal{F} }

\def\G{{\bf G}}

\def\Q{\mathbb{Q}}

\def\G{\mathcal G}

\def\bx{\boldsymbol{x}}
\def\bX{\boldsymbol{X}}

\newcommand{\Comp}{\mathrm{Comp}}

\begin{document}

\title[]{On torsion-free nilpotent loops}

\author[J. Mostovoy]{Jacob Mostovoy}
\address{Departamento de Matem\'aticas, CINVESTAV-IPN\\ Col. San Pedro Zacatenco, Ciudad de M\'exico, C.P.\ 07360\\ Mexico}
\email{jacob@math.cinvestav.mx}

\author[J. M. P\'erez-Izquierdo]{Jos\'e M. P\'erez-Izquierdo}
\address{Departamento de Matem\'aticas y Computaci\'on, Universidad de La Rioja\\ Edificio CCT - C/ Madre de Dios, 53 \\ 26006, Logro{\~n}o, La Rioja \\ Spain}
\email{jm.perez@unirioja.es}

\author[I.P. Shestakov]{Ivan P. Shestakov}
\address{Instituto de Matematica e Estatística, Universidade de S\~ao Paulo,  Caixa Postal 66281, S\~ao Paulo, SP 05311-970, Brazil}
\email{shestak@ime.usp.br}

\thanks{The authors acknowledge the support by the Spanish Ministerio de Ciencia e Innovaci\'on (MTM2013-45588-C3-3-P); J. M. P\'erez-Izquierdo also thanks support by PHBP14/00110; J.\ Mostovoy was also supported by the CONACYT grant 168093-F; I. P. Shestakov also acknowledges support by FAPESP, processo 2014/09310-5 and CNPq, processos 303916/2014-1 and 456698/2014-0.}%

\begin{abstract} We show that a torsion-free nilpotent loop (that is, a loop nilpotent with respect to the dimension filtration) has a torsion-free nilpotent left multiplication group of, at most, the same class. We also prove that a free loop is residually torsion-free nilpotent and that the same holds for any free commutative loop. Although this last result is much stronger than the usual residual nilpotency of the free loop proved by Higman, it is proved, essentially, by  the same method.
\end{abstract}



\maketitle

\section{Introduction}
It has been argued \cite{M1} that the usual definition of nilpotency in loops, given by Bruck \cite{Bruck} is too weak for the most important properties of nilpotent groups to hold. Indeed, contrary to what happens in the associative case, the successive quotients of Bruck's lower central series of a general finitely generated loop are not finitely generated; moreover, they do not carry any algebraic structure similar to that of a Lie ring. 

A different version of the lower central series, designed to eliminate these drawbacks, was introduced in \cite{M1} under the name of the \emph{commutator-associator filtration} (in much greater generality, this definition was later independently stated by Hartl and Loiseau \cite{Hartl}). An important feature of this filtration is its close relationship with the \emph{dimension filtration} which is defined via the powers of the augmentation ideal in the loop ring. By a theorem of Jennings, If $G$ is a group, the dimension filtration of $G$ collapses after $n$th term if and only if $G$ is torsion-free nilpotent of class $n$; the same holds for loops if by nilpotency one understands the nilpotency with respect to the commutator-associator filtration  \cite{M2}. It is therefore reasonable to call a loop \emph{torsion-free nilpotent of class $n$} if its dimension filtration collapses after the $n$th term.

In the present paper we prove two statements that show the usefulness of this notion. Firstly, we show that a torsion-free nilpotent loop of class $n$ has a torsion-free nilpotent left multiplication group of class \emph{at most} $n$. This fails to hold for torsion-free nilpotent loops in the sense of Bruck, a counterexample being the free nilpotent (in the sense of Bruck) loop of class two, whose left multiplication group is not nilpotent. We stress that the nilpotency class of the left multiplication group may be strictly less than the class of the loop; examples of such situation will be given.

Our second (and most important) result says that a free loop is residually torsion-free nilpotent (and, in particular, residually nilpotent with respect to the commutator-associator filtration). This statement is much stronger than the assertion that free loops are residually nilpotent in the sense of Bruck, proved by Higman in \cite{Higman}, although our proof is an application of the same arguments (Higman was well aware that his methods were stronger than what was necessary for the study of Bruck's lower central series).  A similar result is then proved for free commutative loops.

Our main technical tool is the non-associative modification of the Magnus map from the free group on $k$ letters to the group of units in the ring of non-commutative power series in $k$ variables. We will use some very basic facts about loops and about non-associative Hopf algebras as described in \cite{MPS2}, although we shall give some definitions in order to fix the notation.

\section{Preliminaries}

\subsection{Torsion-free nilpotent loops}
A \emph{loop} $L$ is a set with a unital, not necessarily associatve, product, such that both left and right multiplications 
$$L_a\colon L\to L,\quad x\mapsto ax$$ 
$$R_a\colon  L\to L,\quad x\mapsto xa$$ 
by any $a\in L$ are bijective. In a loop one defines the operations of left and the right division by setting
$$a\backslash x = L_a^{-1} x$$
and 
$$x/a = R_a^{-1} x$$
respectively. With each loop $L$ one associates its \emph{left multiplication group} $\mathrm{LMlt} (L)$: this is the subgroup of the permutation group of $L$ considered as a set, which is generated by the $L_a$ for all $a\in L$.

Extending the loop product in a loop $L$ by linearity to the $\Q$-vector space spanned by $L$, one obtains the loop algebra $\Q L$. In the loop algebra, the \emph{augmentation ideal} $I$ is the kernel of the homomorphism $\Q L\to \Q$ which sends each element of $L$ to $1$. The $n$th power $I^n$ of the augmentation ideal is the ideal spanned by the products of at least $n$ elements of $I$. The $n$th \emph{dimension subloop of $L$ over $\Q$} is defined as 
$$D_n L =\{g\in L\,|\, g-1\in I^n\}.$$
We shall refer to these subloops simply as to dimension subloops, without any mention of the field $\Q$. For each $n$ the subloop $D_n L$ is normal (and fully invariant) in $L$ and each quotient $D_n L/D_{n+1} L$ is torsion-free \cite{MP1}. The loop $L$ is called \emph{torsion-free nilpotent of class $n$} if $D_{n+1} L$ is trivial while $D_n L$ is not. 

\subsection{Free loops and the Magnus map}
Let $\bx = \{x_1, \ldots, x_n \}$. The set $W(\bx)$ of \emph{words} on $\bx$ is defined recursively so as to consist of (a) all elements of $\bx$ and $e$; (b) all expressions of the form $uv$, $u\backslash v$ and $u/v$ where $u$ and $v$ belong to $W(\bx)$. The set $W(\bx)$ is the free algebra on $\bx$ with one nullary operation $e$ and three binary operations.  The set $\Comp(w)$ of \emph{components} of a word $w \in W(\bx)$ is defined by setting $$\Comp(e)=\{ e \},$$ $$\Comp(x_i) = \{ e, x_i \}\ \text{when}\ x_i \in \bx$$ and $$\Comp(w) = \Comp(u) \cup \Comp(v) \cup \{w\}\ \text{if}\  w = uv, u\backslash v\ \text{or}\ u/v.$$
A word in $W(\bx)$ is \emph{reduced} if none of its components is of the form
$$u \backslash (uv), u(u \backslash v), (uv)/v,\ (u/v)v,\ u/(v \backslash u),\ (u/v) \backslash u,\  ev,\  e\backslash v,\ u/e,\ ue,\  u/u,\ \text{or}\ v \backslash v.$$ 
The free loop $\F(\bx)$ is the quotient of the algebra $W(\bx)$ by the relations $$eu = u =ue,\ u \backslash (uv) = v = u(u\backslash v) \text{ and } (uv)/v = u = (u/v)v.$$
Each element of $\F(\bx)$ can be represented by a unique reduced word so that $W(\bx)\subset\F(\bx)$; this is a corollary of the Evans Normal Form Theorem \cite{Evans}.

Now, let $\bX = \{X_1, \ldots, X_n \}$. Denote by $\Q\{\bX\}$ the free non-associative algebra on $\bX$ over $\Q$ and let $\overline{\Q \{\bX\}}$ be its completion with respect to the degree; elements of $\overline{\Q\{\bX\}}$ are non-associative formal power series in the $X_i$. The power series with non-zero constant term form a loop  $\overline{\Q\{\bX\}}^{\times}$ under multiplication. Explicitly, the operation of the left division in this loop is determined by
$$(1+B)\backslash (1+A) = 1+(A-B) - B(A-B) + B(B(A-B))-
\ldots$$
for any power series $A$ and $B$ with no constant term. Notice that in the case when $A$ consists only of terms with right-normed parentheses and $B=X_i$ is one of the generators, $(1+B)\backslash (1+A)$ also contains terms with right-normed parentheses only. 

\medskip

The \emph{Magnus map} is the homomorphism
\begin{align*}
\M \colon \F(\bx) &\to \overline{\Q\{\bX\}}^{\times},\\
x_i & \mapsto 1 + X_i.
\end{align*}

Extending the Magnus map by linearity to any $r\in\Q  \F(\bx)$ we see that $\M (r)$ has a zero constant term if and only if $r$ lies in the augmentation ideal of  $\Q \F(\bx)$. This implies the following:
\begin{lem}[\cite{MP1}]
An element $w\in \F(\bx)$ lies in $D_n\F(\bx)$ if and only if 
$$\M(w) = 1+ \mathrm{terms\ of\ degree\ } n\ \mathrm{and\ higher}.$$
\end{lem}
In particular, in order to prove that $\F(\bx)$ is residually torsion-free nilpotent, it is sufficient to show that the Magnus map is injective. This is totally analogous to the associative case where the Magnus map is used to prove that the free group is residually torsion-free nilpotent \cite{MKS}. However, while in the associative case this is established by a straightforward argument involving the syllable length of $w$, the non-associative version is somewhat more complex.

The Magnus map can be defined, in an entirely analogous fashion, for free commutative loops. Here the algebra $\overline{\Q\{\bX\}}$ should be replaced by the free commutative non-associative algebra on $\bX$. We shall prove that the commutative Magnus map is also injective, and this will imply that free commutative loops are residually torsion-free nilpotent.

\section{The left multiplication group}

\begin{thm} \label{main}
If $L$ is a torsion-free nilpotent loop of class $n$, the group $\mathrm{LMlt} (L)$  is nilpotent of class at most $n$. 
\end{thm}
We shall see that there are examples where the class of $\mathrm{LMlt} (L)$ is strictly smaller than the class of $L$. On the other hand, when $L$ is a group, we have $\mathrm{LMlt} (L)=L$ and, hence, the class of  $\mathrm{LMlt} (L)$ is the same as the class of $L$.

\medskip
The key to the proof of Theorem~\ref{main} is the following:
\begin{lem}\label{first}
For all $a\in L$ and $F\in\gamma_n \mathrm{LMlt} (L)$ we have $F(a)\equiv a \mod D_n L.$ 
\end{lem}
\begin{proof}
First, let us consider the case when $L$ is the free loop generated by $a, x_1,\ldots,$ $x_n$,  and $F$ is the $n$-fold commutator $[L_{x_1},[\ldots, [L_{x_{n-1}},L_{x_n}]...]]$. Then $\M(F(a))$ contains right-normed terms only. By dropping the parentheses we get the associative Magnus map of the element $$[{x_1},[\ldots, [{x_{n-1}},{x_n}]...]]\cdot a$$ in the free \emph{group} generated by $a, x_1,\ldots, x_n$; this series is of the form $$1+a+\text{terms of degree}\ n\ \text{and greater}.$$ This implies that $\M(F(a))$ is also of the same form. Indeed, the homomorphism of the free non-associative algebra on a certain set of generators onto the free associative algebra on the same generators which ``forgets the parentheses'' maps the space of the right-normed monomials isomorphically onto the free associative algebra.
This proves the lemma in this case. The exact same argument works in the case when $L$ is the free loop generated by $a, x_1,\ldots, x_m$, with $m\geq n$, and $F$ is a product of arbitrary $n$-fold commutators in the $L_{x_i}$.

Now, each element $F(a)$ with $F\in\gamma_n \mathrm{LMlt} (L)$ in an arbitrary loop $L$ is a homomorphic image of $\widetilde{F} (a)$ where $\widetilde{F}$ is a product of arbitrary $n$-fold commutators in the $L_{x_i}$ in the free loop generated by  $a, x_1,\ldots, x_m$, for some $m$. Since the dimension subloops are respected by loop homomorphisms, the lemma follows.
\end{proof}

Theorem~\ref{main} is now an easy consequence of Lemma~\ref{first}. Indeed, if $L$ is torsion-free nilpotent  of class $n$, we have $D_{n+1} L= \{1\}$. Then, by Lemma~\ref{first}, for $F\in\gamma_{n+1} \mathrm{LMlt} (L)$ and any $a\in L$ we have $F(a)=a$, which means that $F$ is the identity in $\mathrm{LMlt} (L)$. 
\medskip
Now, let us consider several examples.
\begin{prop}\label{freecomm}
Let $L$ be the commutative loop whose elements are pairs of integers $(p,q)$ with
$$(p,q)\cdot (p',q') = \left(p+p', q+q'+\binom{p}{2}\binom{p'}{2}\right).$$ 
Then $L$ is torsion-free nilpotent of class 4, with 
$$D_2 L=D_3 L  = D_4 L = \{(0,q)\,|\, q\in \Z\},$$ while $\mathrm{LMlt} (L)$ is nilpotent of class 3.
\end{prop}
\begin{proof}
From what follows it should become clear that $L$ is the reduction of the free commutative loop on one generator modulo $D_5$ and the argument is based on the Magnus map.

Let $L_2 = \{ (0,q)\, |\, q\in\Z\} \subset L$. It is clear that $L_2$ is a normal subloop of $L$ (in fact, it is the centre of $L$) and the quotient $L/L_2$ is abelian and torsion-free. It follows that $D_i L \subset L_2$ for $i>1$. 

It can be seen directly that $(0,1)\in D_4 L$. Indeed, if $X=(1,0)-1\in I$, we have
$$(0,1) - 1 = (4,0) \backslash ((4,1)-(4,0)) = (4,0) \backslash (X^2X^2 - X^3 X),$$
which is readily seen to lie in $I^4$; see also Lemma~1 of \cite{MP1}. (We can write $X^3$ here since, by commutativity, $X^2X=XX^2$). Since $(0,1)$ generates $L_2$, we see that 
$D_4 L = L_2.$

Let us now show that $D_5 L$ is trivial. 
Let ${\Q_c\{X\}}$ be the free commutative non-associative algebra on one generator $X$ and $\mathcal{I}^k$ the ideal consisting of non-associative polynomials in $X$ without terms of degree $<k$. As a a vector space,  ${\Q_c\{X\}}/\mathcal{I}^5$  is spanned by the monomials $1,X,X^2,X^3,X^3X$ and $X^2X^2$. Consider the map $$L \to {\Q_c\{X\}}/\mathcal{I}^5$$
given by
$$(p,q) \mapsto 1+pX+\binom{p}{2}X^2+\binom{p}{3}X^3+\left(\binom{p}{4}-q\right) X^3X + q X^2 X^2  \mod \mathcal{I}^5.$$
A direct computation shows that it is an injective homomorphism of $L$ into the loop of units in ${\Q_c\{X\}}/\mathcal{I}^5$. Extend it to the loop algebra $\Q L$ by linearity; the augmentation ideal  $I \subset \Q L$ is then mapped to the ideal $\mathcal{I}/\mathcal{I}^5$.  In particular, if $g\in D_5 L$, we see that $g-1$ maps to zero, which implies that $g=1$.

It remains to see that $\mathrm{LMlt} (L)$ is nilpotent of class three. On one hand, notice that $L_{(p,q)} = L_{(p,0)}L_{(0,q)}$ and that $L_{(0,q)}$ lies in the centre of  $\mathrm{LMlt} (L)$ for all $q$. On the other hand, we have 
$$[L_{(p,0)},L_{(r,0)}](a,b) = L_{(p,0)}^{-1}L_{(r,0)}^{-1}L_{(p,0)}L_{(r,0)} (a,b)= (a, b+\frac{1}{2} apr(p-r)).$$
Furthermore, 
$$[L_{(s,0)}, [L_{(p,0)},L_{(r,0)}]](a,b) = (a, b-\frac{1}{2} prs(p-r))).$$
It follows that $[L_{(s,0)}, [L_{(p,0)},L_{(r,0)}]]$ lies in the center of $\mathrm{LMlt} (L)$ and, hence, $\mathrm{LMlt} (L)$ is three-step nilpotent.
\end{proof}

\begin{prop}
Let $L$ be the loop whose elements are pairs of integers $(p,q)$ with
$$(p,q)\cdot (p',q') = \left(p+p', q+q'+\binom{p}{2} p'\right).$$ 
Then $L$ is torsion-free nilpotent of class 3, with 
$$D_2 L=D_3 L= \{(0,q)\,|\, q\in \Z\},$$ while $ \mathrm{LMlt} (L)$ is nilpotent of class 2.
\end{prop}
The proof of this proposition is very similar to that of Proposition~\ref{freecomm}, only in this example the loop in question is isomorphic to the reduction of the free \emph{non-commutative} loop on one generator modulo $D_4$; we omit it.

\medskip

Finally, we have the following fact which shows the radical difference between nilpotency in the sense of Bruck for torsion-free loops and torsion-free nilpotency for loops:
\begin{prop}\label{nonnilpotent}
Let $L = F/F_3$ be the free nilpotent (in the sense of Bruck) loop of class 2 on a countable number of generators. Then $\mathrm{LMlt} (L)$ is not nilpotent.
\end{prop}
We shall prove this assertion in the next section. Note that, by \cite{Wright}, for any nilpotent, in  the sense of Bruck, loop $L$ the left multiplication group $\mathrm{LMlt} (L)$ is solvable.

\section{Higman's construction}
Here we state several results of \cite{Higman} which will be of importance.

\medskip

Let $F=\F\{\bx\}$ be the free loop  on the set $\bx = \{x_1, \ldots, x_n \}$, $\alpha \colon F\to L$ a surjective homomorphism and $N=\ker\alpha$. The subloop $[N, F]$ is defined as smallest subloop of $N$ which is normal in $F$ and such that $N/[N,F]$ is in the centre of $F/[N,F]$. In particular, if $N=F_k$, the $k$th term of Bruck's lower central series, then, by definition, $[N,F]=F_{k+1}$.

\medskip

Define the abelian group $A$ to be freely generated by two types of symbols: 
\begin{itemize}
\item $x_i$ for each $x_i\in \bx$;  
\item $\langle l_1,  l_2 \rangle$ for each pair $l_1, l_2\in L\setminus\{e\}$.
\end{itemize}
On the set $L\times A$ consider the following product:
$$(l_1, a_1)(l_2, a_2) = (l_1 l_2, a_1+a_2+ \langle l_1,  l_2 \rangle),$$
where $\langle l,  e \rangle = \langle e,  l \rangle = 0$ for all $l\in L$.
With this product, $L\times A$ becomes a loop, which we denote by $(L,A)$. In particular,
$$(l_1, a_1)/(l_2, a_2) = (l_1 / l_2, a_1 - a_2 - \langle l_1/l_2, l_2 \rangle),$$
$$(l_2, a_2)\backslash(l_1, a_1) = (l_2 \backslash l_1, a_1 - a_2 - \langle l_2, l_2 \backslash l_1 \rangle).$$
Consider the homomorphism 
\begin{align*}
\delta\colon F&\to (L,A)\\
w&\mapsto (\alpha(w), \psi(w))
\end{align*}
where $\psi (x_i) =  x_i\in A$.
\begin{lem}[\cite{Higman}, Lemma 3]\label{H3}
$$\ker(\delta) = [N, F]. $$ 
\end{lem}
This lemma has the following corollary: 
\begin{lem}
$\gamma_{n-1}\mathrm{LMlt} (F/F_3)$ is non-trivial.
\end{lem}
\begin{proof}
Let $N =[F,F]$ so that $L=F/N$ is a free abelian group on $x_1,\ldots, x_n$ (strictly speaking, we should write $\alpha(x_i)$ instead of $x_i$ but the notation can take some abuse)  and $[N, F] =F_3$.

By Lemma \ref{H3} the image of the map $\delta\colon F\to (L,A)$ is isomorphic to $F/F_3$.
 Write
$$y = [L_{x_{n}}, [L_{x_{n-1}}, \ldots [L_{x_{3}}, L_{x_{2}}]]] (x_1).$$ 
Then $\delta (y) = (x_1, a)$ with $a\neq x_1$. Indeed, it is easy to see (by induction on $n$, for instance) that the term $\langle x_2, x_1 \rangle$ appears in $a$ with coefficient 1. 
This implies that the iterated commutator $[L_{x_{n}}, [L_{x_{n-1}}, \ldots [L_{x_{3}}, L_{x_{2}}]]]$ is not the identity and, hence, $\gamma_{n-1}\mathrm{LMlt} (F/F_3)\neq \{ 1\}.$
\end{proof}
Proposition \ref{nonnilpotent} follows.

\medskip

The main technical result of Higman's paper \cite{Higman} is the following:
\begin{lem}[\cite{Higman}, Lemma 6]\label{H6}
Let $w,w' \in W(\bx)$ be reduced words such that
\begin{itemize}
		\item $w' \not\in \Comp(w)$,
		\item $\alpha(w) = \alpha(w')$,
		\item if $u, v$ are words in $\Comp(w) \cup \Comp(w')$ such that $\alpha(u) = \alpha(v)$ then either $u = v$ or $\{ u,v\} = \{w,w'\}$.
\end{itemize}
Then there exists a generator $x_i$ or $\langle l_1, l_2\rangle$ of ${A}$ whose coefficient in $\psi(w)$  is zero, and in  $\psi(w')$ is $\pm 1$.
\end{lem}

This statement has a consequence which will be of immediate importance for us: 

\begin{lem}[\cite{Higman}, Corollary 1]\label{HC1}
Let $S$ be a finite set of reduced words containing together with any word all of its components. If the images of the elements of $S$ in $F/N$ under the natural projection map are not all distinct, the number of elements in the image of $S$ in $F/[N,F]$ is strictly greater than the number of elements in the image of $S$ in $F/N$. 
\end{lem}

\section{Injectivity of the Magnus map}

Let $\F_c(\bx)$ be the free commutative loop on the set $\bx=\{x_i\}$ and $\Q_c\{\bX\}$ the free commutative non-associative algebra on the corresponding set $\bX=\{X_i\}$. The commutative version of the Magnus map sends $\F_c(\bx)$ into the loop of invertible elements in the completion of $\Q_c\{\bX\}$ and is defined exactly in the same way as the the usual Magnus map: $x_i$ is sent to $1+X_i$.

\begin{thm} The Magnus maps $$\M\colon \F(\bx) \to \overline{\Q\{\bX\}}^{\times}$$ 
and
 $$\M_c\colon \F_c(\bx) \to \overline{\Q_c\{\bX\}}^{\times}$$ 
are injective.
\end{thm}
In particular, the free loops $\F(\bx)$ and $\F_c(\bx)$ are residually torsion-free nilpotent.

\subsection{Group-like elements and the modified Magnus map}
Recall that in a Hopf algebra a \emph{primitive element} is an element $x$ satisfying $$\Delta (x) = x\otimes 1 + 1\otimes x$$ and a \emph{group-like element} is an element $g$ such that $$\Delta (g)= g\otimes g$$ 
and $\epsilon(g) = 1$.  The set $\G(H)$ of all group-like elements in a non-associative Hopf algebra $H$ is, actually, a loop \cite{PI2}.

The free algebra $\Q\{\bX\}$ is a Hopf algebra whose coproduct is defined by the condition that the generators are primitive. The only group-like element in $\Q\{\bX\}$ is 1;  however, the completion $\overline{\Q\{\bX\}}$ has many group-like elements. A group-like element $e(X)$ in the algebra $\overline{\Q\{X\}}$ of power series in one non-associative and non-commutative variable $X$ with a non-zero coefficient at $X$ is called a \emph{base for logarithms} or a \emph{non-associative exponential series} (see  \cite{MPS1, MPS3}).
An example of a base for logarithms is the usual exponential series endowed with the right-normed parentheses. Given $e(X)$ there exists another series $\log_e(1+X)$ such that $\log_e(e(X))=X$ and $e(\log_e(1+X))=X$. 

\medskip

It will be convenient to use a slightly different version of the Magnus map by choosing a base $e(X)$ for logarithms and setting
\begin{align*}
\M'\colon \F(\bx) &\to \G(\overline{\Q\{\bX\}})\subset \overline{\Q\{\bX\}}^{\times},\\
x_i&\mapsto e(X_i).
\end{align*}
The map $X_i \mapsto e(X_i)-1$ determines an algebra homomorphism  $\overline{\Q\{\bX\}}\to \overline{\Q\{\bX\}}$ which sends $\M(w)$ to $\M'(w)$ for any $w\in F= \F(\bx)$. Therefore, if $\M'$ is injective so is $\M$.

\medskip

Let $N=\ker{\M'}$ and $L=F/N$. We shall prove that $N$ coincides with the kernel of Higman's homomorphism $F\to (L,A)$ and this, by Lemmas~\ref{H3} and \ref{HC1}, will imply that $N$ is trivial. For this purpose, we shall embed $(L, A)$ into a bigger loop, namely the loop of group-like elements of a Hopf-algebraic version of Higman's construction $(L,A)$.

\subsection{Higman's construction for Hopf algebras}

Let ${m}(\bX)$ be the set of all non-associative monomials on $\bX$ of degree $\geq 1$ and $T$ the set of symbols
$$\{t_1, \ldots, t_n\} \sqcup \{ t(m_1,m_2) \,|\, m_1,m_2\in m(\bX)\}$$
with degrees
$$|t_1| = \ldots = |t_n| =1 \quad\text{and}\quad | t(m_1,m_2) | = |m_1| + |m_2|,$$
where $|X_1| = \ldots | X_n| =1$. Write $\Q[T]$ for the usual commutative and associative algebra on $T$ with the structure of a Hopf algebra obtained by declaring all elements of $T$ to be primitive.

Define a linear map $$t\colon \Q\{\bX \}\otimes \Q\{\bX \} \to \mathrm{span}{\langle T \rangle}$$ 
by setting $t(m_1\otimes m_2) = t(m_1, m_2), t(m_1 \otimes 1) = t(1 \otimes m_2) = t(1 \otimes 1) = 0$ and let 
$$t^*\colon  \Q\{\bX \}\otimes \Q\{\bX \} \to \Q[T]$$
be the coalgebra morphism
$$t^*(\mu\otimes \mu') = \epsilon(\mu)\epsilon(\mu') 1 +\sum_{k\geq 1} \frac{1}{k!} t(\mu_{(1)}, \mu'_{(1)}) \ldots t(\mu_{(k)}, \mu'_{(k)}).$$
We shall use the same notation $t$ and $t^*$ for the extensions of these maps to maps between the respective completions of ${\Q\{\bX \}\otimes \Q\{\bX \}}$, ${\mathrm{span}{\langle T \rangle}}$ and  ${\Q[T]}$ with respect to the degree.

The maps $t$ and $t^*$ are related in the following way: for any $g,g' \in \G(\overline{\Q[\bX]})$ we have
$$t^*(g \otimes g') = \sum_{k\geq 0}^{\infty}  \frac{1}{k!} t(g \otimes g')^k = \exp (t(g\otimes g')) \in \overline{\Q[T]},$$
where $\exp$ is the usual exponential series.
\begin{lem}\label{A}
The elements $\exp(t_1), \ldots \exp(t_n)$ and $t^*(g\otimes g')$, where $g, g' \in  \G(\overline{\Q[\bX]}) \backslash \{1\}$, freely generate a multiplicative abelian subgroup of $\G(\overline{\Q[T]})$.
\end{lem}
\begin{proof}
Assume that $$\exp(t_1)^{e_1} \cdots \exp(t_n)^{e_n} {t}^*(g_1\otimes g'_1)^{d_1} \cdots {t}^*(g_m\otimes g'_m)^{d_m}  = 1$$	
for some  $g_1,g'_1,\dots, g_m,g'_m \in \mathcal{G}(\overline{\Q\{\bX\}}) \setminus\{1\}$ with $(g_i,g'_i) \neq (g_j,g'_j)$ if $i\neq j$, and some integers $e_1,\dots, e_n, d_1, \dots, d_m$.  Then 
$$\exp(e_1 t_1 + \cdots + e_n t_n + d_1 t(g_1\otimes g'_1) + \cdots + d_m t(g_m\otimes g'_m)) = 1$$ 
and 
$$e_1 t_1 + \cdots e_n t_n + d_1 t(g_1\otimes g'_1) + \cdots + d_m t(g_m\otimes g'_m) = 0.$$
Since $\{t_1,\dots,t_n\}$ and $\{ t(m_1\otimes m_2) \mid m_1,m_2 \in m(\bX) \}$ are algebraically independent, we have $e_1 = \cdots = e_n = 0$ and 
$$t(d_1 g_1 \otimes g'_1 + \cdots + d_m g_m \otimes g'_m) = d_1 t(g_1\otimes g'_1) + \cdots + d_m t(g_m\otimes g'_m) = 0.$$
The kernel of $t \colon \overline{ \Q\{ \bX \} \otimes \Q\{ \bX \}} \rightarrow \overline{ \mathrm{span}\langle T \rangle }$ is 
$$\overline{\Q\{ \bX \}} \otimes 1 + 1 \otimes \overline{\Q\{ \bX \}}$$ 
so that
$$d_1 \cdot g_1 \otimes g'_1 + \cdots + d_m \cdot g_m \otimes g'_m = \mu \otimes 1 + 1 \otimes \eta$$ for some $\mu, \eta \in \overline{ \Q\{\bX\}}$. As $g_1,g'_1,\dots, g_m,g'_m \in \mathcal{G}(\overline{\Q\{\bX\}}) \backslash \{1\}$ and since different group-like elements are linearly independent \cite[Theorem 2.1.2]{Ab04}, we see that  $d_1 = \cdots = d_m = 0$.
\end{proof}

Now, define a new product on $ \Q\{ \bX \}  \otimes  \Q[T]$ by
\begin{equation}\label{product}
(x\otimes \alpha) (y\otimes \beta) = \sum x_{(1)} y_{(1)} \otimes t^*(x_{(2)} \otimes y_{(2)}) \alpha\beta.
\end{equation}
With this product, both $ \Q\{ \bX \}  \otimes  \Q[T]$ and $ \overline{\Q\{ \bX \}  \otimes  \Q[T]}$ are connected bialgebras.
Consider the homomorphism
\begin{align*}
\widetilde{\M}'\colon \F(\bx)&\to \G( \overline{\Q\{ \bX \}  \otimes  \Q[T]}) \\
x_i &\mapsto e(X_i) \otimes \exp(t_i),
\end{align*}
where $e$ is the same base for logarithms used for defining $\M'$.
The formula $$X_i \mapsto\log_e(e(X_i) \otimes \exp(t_i))$$ defines  an injective algebra homomorphism
$$\phi \colon  \overline{\Q\{ \bX \}}  \to  \overline{\Q\{ \bX \}  \otimes  \Q[T]}$$ 
and we have $\phi\circ \M' = \widetilde{\M}'.$
In particular,   for any words $w_1, w_2\in \F(\bx)$ the equality $$\M'(w_1)= \M'(w_2)$$ holds if an only if $\widetilde{\M}'(w_1)= \widetilde{\M}'(w_2)$. 

Now, by definition, $L$ is a subloop of  $\G(\overline{\Q\{\bX\}})$ and  Lemma~\ref{A} gives an natural injective homomorphism of $(L,A)$ into $\G( \overline{\Q\{ \bX \}  \otimes  \Q[T]})$.
Moreover,  on $L$ the map $\phi$ coincides with the map $L\to (L,A)$ and, therefore, as mentioned before, Lemmas~\ref{H3} and \ref{HC1} imply that $\M'$ is injective.

\subsection{The commutative case}
The proof of the injectivity of ${\mathcal{M}}_c$ involves minor changes:
\begin{itemize}
	\item denote by $m_c(\bX)$  the set of all monomials of degree $\geq 1$ in the commuting variables $X_1,\dots, X_n$;
	\smallskip
	
	\item define the set $T$  as $\{t_1,\dots, t_n\} \sqcup \{t(m_1, m_2) \mid m_1,m_2 \in m_c(\bX) \}$ where we assume that $t(m_1,m_2) = t(m_2,m_1)$ for all $m_1, m_2 \in m_c(\bX)$;
		\smallskip
		
	\item define the coalgebra morphism $$t^* \colon \Q_c\{\bX\} \otimes \Q_c\{\bX\} \rightarrow \Q[T]$$ as before. The symmetry of $t$ implies $t^*(\mu \otimes \mu') = {t}^*(\mu' \otimes \mu)$
	for all $\mu, \mu' \in \Q_c\{\bX\}$;
		\smallskip
		
	\item  write $\Q_c\{\bX\} \otimes \Q[T]$ for the commutative bialgebra with the product given by \eqref{product};
		\smallskip
		
	\item note that Lemma~\ref{A} trivially holds when $\overline{\Q\{\bX\}}$ is replaced by $\overline{\Q_c\{\bX\}}$;
		\smallskip	
	\item let $\M'_c\colon \F_c\{\bx\} \rightarrow \mathcal{G}(\overline{\Q_c\{\bX\}})$ and $\widetilde{\M}'_c\colon \F_c\{\bx\} \rightarrow \mathcal{G}(\overline{\Q_c\{\bX\} \otimes \Q[T]})$ be the homomorphisms determined by $\M'_c(x_i)= e(X_i)$ and $\widetilde{\M}'_c(x_i)= e(X_i) \otimes \exp(t_i)$ for some base for logarithms $e(X)$;
	%
		\smallskip
		
	\item fix a total order on $W(\bx)$ and call a word in $W(\bx)$ \emph{reduced}  if none of its components is the left-hand side of any of the equations
	\begin{displaymath}
		 \begin{array}{lll}
			  uv = vu \text{ $(u < v)$} \quad& u / v = v \backslash u \qquad& (v \backslash u) \backslash u = v\\
			 v \backslash (uv) = u & u \backslash (uv) = v  &  (u \backslash v) u = v	   \\
			   ev = v & u(u \backslash v) = v & e \backslash v = v  \\
				ue = u   &   & v \backslash v = e;\\ 
		 \end{array}
	\end{displaymath}
		\smallskip
	\item	in the definition of the loop $(L,A)$, take the symbols $\langle l_1, l_2\rangle$ to be symmetric in the sense that $\langle l_1, l_2\rangle = \langle l_2, l_1\rangle$.
	\smallskip
\end{itemize}
With these new definitions Lemma~\ref{H6} (that is, Higman's Lemma~6) remains true, and the injectivity of $\mathcal{M}_c$ follows from it in the very same fashion as in the non-commutative case. For the sake of completeness, we include the proof of the commutative version of Lemma~\ref{H6} here, although it follows Higman's original proof very closely. 

\subsection{Proof of the commutative version of Lemma~\ref{H6}}

First, we observe that $w' \neq e$ since $w'\not\in \Comp(w)$. 

The case when $w' = x_i$ is straightforward: $x_i = w' \not\in \Comp(w)$ implies that the coefficient of $x_i$ in $\psi(w')$ is $1$ while in $\psi(w)$ it is $0$. Therefore, we only have to consider two cases: (1) $w' = uv$ and (2) $w' = u \backslash v$, where both $u$ and $v$ are reduced words. As $w'$ is reduced, we have $u \neq e$. In fact, $\alpha(u)\neq 1$, the unit element of $L$, since $\alpha(u) = 1 = \alpha(e)$ implies, by hypothesis, that either $u = e$, a contradiction, or $\{u,e\} = \{w,w'\}$, which is not possible since $w' \neq u,e$. 

In what follows, we will often consider two elements, say $c_1$ and $c_2$, in $\Comp(w') \cup \Comp(w)$ with $\alpha(c_1) = \alpha(c_2)$. By hypothesis, this implies that either $c_1 = c_2$ or $\{w',w\} = \{c_1,c_2\}$. It will be clear that $w' \neq c_1, c_2$  (usually $c_1$ and $c_2$ will be either $e$, components of $w$ or components of $w'$ different from $w'$) so we will conclude that $c_1 = c_2$ without further explanation.

\subsubsection{Case $w' = uv$.}

First we observe that $\alpha(v) \neq 1$. Indeed, $\alpha(v) = 1 = \alpha(e)$ would imply that $v = e$ so that $w' = ue$ is not reduced, a contradiction. This ensures that $\langle\alpha(u),\alpha(v)\rangle$ is a generator of ${A}$. We will prove that the coefficient of $\langle\alpha(u),\alpha(v)\rangle$ in $\psi(\tau)$ is zero for any $\tau \in \{w,u,v\}$, which will allow us to conclude that the coefficient of $\langle\alpha(u),\alpha(v)\rangle$ in $\psi(w)$ is zero while in $\psi(w')$ it equals to one. To this end,  assume that the coefficient of $\langle\alpha(u),\alpha(v)\rangle$ in $\psi(\tau)$ is non-zero and observe that $w' \not\in \Comp(\tau)$. There are two possibilities:
\begin{enumerate}
\item {$\tau$ has a component $\tau_1 \tau_2$ with $\{\alpha(\tau_1),\alpha(\tau_2)\} = \{ \alpha(u),\alpha(v)\}$.} 
\begin{itemize}
	\item If $\alpha(\tau_1) = \alpha(u)$ and $\alpha(\tau_2) = \alpha(v)$ then $\tau_1 = u$ and $\tau_2 = v$. Thus $\tau = uv = w'$, a contradiction.
	\item If $\alpha(\tau_1) = \alpha(v)$ and $\alpha(\tau_2) = u$ then $\tau_1 = v$, $\tau_2 = u$ and $\tau = vu$. Both $\tau$ and $w'$ are reduced so $u \geq v$ and $v \geq u$. Hence $u = v$ and $\tau = uu = w'$, a contradiction.
\end{itemize}
\item  {$\tau$ has a component $\tau_1 \backslash \tau_2$ with $\{\alpha(\tau_1),\alpha(\tau_1\backslash \tau_2)\} = \{ \alpha(u), \alpha(v)\}$.} 
\begin{itemize}
	\item If $\alpha(\tau_1) = \alpha(u)$ and $\alpha(\tau_1  \backslash \tau_2) = \alpha(v)$ then $\tau_1 = u$ and 
	\begin{itemize}
		\item either $\tau_1 \backslash \tau_2 = v$, which implies that $w' = uv = \tau_1(\tau_1 \backslash \tau_2)$ is not reduced, a contradiction, 
		\item or $w' = \tau_1 \backslash \tau_2$ and $w = v$, which is again a contradiction since $w' \not\in \Comp(\tau)$.
	\end{itemize}
	\item If $\alpha(\tau_1) = \alpha(v)$ and $\alpha(\tau_1 \backslash \tau_2) = \alpha(u)$ then $\tau_1 = v$ and $\tau_1 \backslash \tau_2 = u$. Thus $w' = (\tau_1 \backslash \tau_2) \tau_1$ is not reduced, a contradiction.
\end{itemize}
\end{enumerate}

\subsubsection{Case $w' = u \backslash v$.} 
We first observe that $\alpha(u \backslash v) \neq  1$. Indeed, $\alpha(u \backslash v) = 1$ would imply $\alpha(u) = \alpha(v)$ so that $u= v$ and $w' = u \backslash u$ is not reduced, a contradiction. This ensures that $\langle\alpha(u),\alpha(u\backslash v)\rangle$ is a generator of ${A}$. We will prove that the coefficient  of $\langle\alpha(u),\alpha(u\backslash v)\rangle$ in $\psi(\tau)$ is zero for any $\tau \in \{w, u,v\}$, which will allow us to conclude that the coefficient of $\langle\alpha(u),\alpha(u\backslash v)\rangle$ in $\psi(w)$ is zero while it is $-1$ in $\psi(w')$. To this end, assume that the coefficient of $\langle\alpha(u),\alpha(u\backslash v)\rangle$ in $\psi(\tau)$ is non-zero and observe that $w' \not\in \Comp(\tau)$. There are two possibilities:
\begin{enumerate}
	\item {$\tau$ has a component $\tau_1\tau_2$ with $\{\alpha(\tau_1),\alpha(\tau_2)\} = \{\alpha(u),\alpha(u\backslash v) \}$.}
	\begin{itemize}
		\item  If $\alpha(\tau_1) = \alpha(u)$ and $\alpha(\tau_2) = \alpha(u \backslash v)$  then $\tau_1 = u$ and $\alpha(v) = \alpha(u)\alpha(w') = \alpha(u)\alpha(\tau_2) = \alpha(\tau_1\tau_2)$. Thus 
		\begin{itemize}
			\item either $v = \tau_1\tau_2$, which implies $w' = \tau_1 \backslash (\tau_1\tau_2)$, a contradiction, 
			\item or $w' = \tau_1\tau_2$ and $w= v$, which is not possible since $w' \not\in \Comp(\tau)$. 
		\end{itemize}
		\item If $\alpha(\tau_1) = \alpha(u \backslash v)$ and $\alpha(\tau_2) = \alpha(u)$ then $\tau_2 = u$ and $\alpha(\tau_1\tau_2) = \alpha(\tau_1)\alpha(\tau_2) = \alpha(\tau_2)\alpha(\tau_1) = \alpha(u)\alpha(u \backslash v) = \alpha(v)$. Thus, $\tau_1\tau_2 = v$ and $w' = u \backslash v = \tau_2 \backslash (\tau_1 \tau_2)$ is not reduced, a contradiction.
	\end{itemize}
	\item {$\tau$ has a component $\tau_1 \backslash \tau_2$ with $\{\alpha(\tau_1), \alpha(\tau_1 \backslash \tau_2)\} = \{ \alpha(u), \alpha(u \backslash v)\}$.} 
	\begin{itemize}
		\item If $\alpha(\tau_1) = \alpha(u)$ and $\alpha(\tau_1 \backslash \tau_2) = \alpha(u \backslash v)$ then $\tau_1 = u$ and $\alpha(v) = \alpha(u)\alpha(u \backslash v) = \alpha(\tau_1)\alpha(\tau_1 \backslash \tau_2) = \alpha(\tau_2)$ so $\tau_2 = v$. Hence $w' = u \backslash v = \tau_1 \backslash \tau_2 \in \Comp(\tau)$, a contradiction. 
		\item If $\alpha(\tau_1) = \alpha(u \backslash v)$ and $\alpha(\tau_1 \backslash \tau_2) = \alpha(u)$ then $u = \tau_1 \backslash \tau_2$ and $\alpha(v) = \alpha(u)\alpha(u \backslash v) = \alpha(u \backslash v) \alpha(u) = \alpha(\tau_1)\alpha(\tau_1 \backslash \tau_2) = \alpha(\tau_2)$. Thus, $\tau_2 = v$ and $w' = u \backslash v = (\tau_1 \backslash \tau_2) \backslash \tau_2$ is not reduced, a contradiction.
	\end{itemize}
\end{enumerate}

\end{document}